\documentclass[11pt,reqno]{amsart}
\usepackage{amssymb,xfrac}
\usepackage{amsmath}
\usepackage{xcolor}
\usepackage{graphics, cite, bookmark}
\usepackage{epsfig}
\usepackage{hyperref}
\usepackage{epstopdf}
\usepackage{ esint }
\usepackage{a4wide}
\usepackage{amsthm}
\usepackage{lmodern}

\newtheorem{theorem}{Theorem}[section]
\newtheorem*{theorem*}{Theorem}
\newtheorem{lemma}[theorem]{Lemma}

\newtheorem{proposition}[theorem]{Proposition}

\theoremstyle{definition}
\newtheorem{definition}{Definition}[section]

\theoremstyle{remark}
\newtheorem{remark}{Remark}[section]

\numberwithin{equation}{section}
\allowdisplaybreaks

\newcommand{\R}{\ensuremath{\mathbb{R}}}
\newcommand{\br}{\ensuremath{\mathbf{r}}}
\newcommand{\N}{\ensuremath{\mathbb{N}}}

\newcommand{\C}{\ensuremath{\mathbb{C}}}
\renewcommand{\S}{\ensuremath{\mathbb{S}}}

\newcommand{\sym}{\ensuremath{\mathrm{sym}}}

\renewcommand*{\div}{\textrm{div}}
\usepackage{bbm}

\newcommand{\norm}[1]{\left\lVert#1\right\rVert}

\newcommand\restrl{\mathop{\hbox{\vrule height .3pt width 5pt depth 0pt\vrule height 7pt width .3pt depth 0pt}}\nolimits}

\title[Isometric immersion and Darboux equation]{Isometric immersions and weak solutions to the Darboux equation}

\author{Wentao Cao, Jonas Hirsch, and Dominik Inauen}

\address{Wentao Cao, Academy for Multidisciplinary Studies, Capital Normal University, West 3rd Ring North Road 105, Beijing, 100048 P.R. China. E-mail:{\tt cwtmath@cnu.edu.cn}}
\address{Jonas Hirsch, Institut f\"{u}r Mathematik, Universit\"{a}t Leipzig, D-04109, Leipzig, Germany.  E-mail:{\tt jonas.hirsch@math.uni-leipzig.de}}
\address{Dominik Inauen, Institut f\"{u}r Mathematik, Universit\"{a}t Leipzig, D-04109, Leipzig, Germany.  E-mail:{\tt dominik.inauen@math.uni-leipzig.de}}

\begin{document}
\begin{abstract}
    We study the Darboux equation, a fundamental PDE arising in the theory of isometric immersions of two-dimensional Riemannian manifolds into $\mathbb{R}^3$, in the low-regularity regime. We introduce a notion of weak solution for $u\in C^{1,\theta}$ with $\theta>1/2$, and show that the classical correspondence between solutions of the Darboux equation and isometric immersions remains valid in this regime. The key ingredient is an extension of the classical flatness criterion to Hölder continuous metrics, achieved via an analysis of a weak notion of Gaussian curvature.
\end{abstract}
\maketitle

\section{Introduction}

It is well-known that sufficiently regular (e.g., $C^2$ or $C^3$) isometric immersions (i.e., length-preserving immersions from Riemannian manifolds $(M,g)$ into Euclidean space) automatically satisfy higher-order equations that relate intrinsic to extrinsic geometry, such as the Gauss equation or the Codazzi-Mainardi system. One such equation, which serves as the central object of study in this paper, is the Darboux equation
\begin{equation}\label{eq-darb}
\det \nabla_g^2 u = K_g(1-|du|_g^2).
\end{equation}
Here, $g$ is a $C^2$ Riemannian metric on a domain $\Omega \subset \mathbb{R}^2$, $K_g$ denotes its Gaussian curvature, $u:\Omega\to\mathbb{R}$ is a $C^2$ function, and $\nabla_g^2u$ is the ``Riemannian Hessian" of $u$, i.e., the $(1,1)$-tensor obtained from the covariant Hessian $\nabla du$ by raising an index,  see Section \ref{s:hessian}.

Equation \eqref{eq-darb} was first derived by Darboux in 1894 \cite{Darboux} (see also \cite{Hong1999, Khuri2007} for later developments and related curvature equations). The connection between \eqref{eq-darb} and isometric immersions is given by the following correspondence (see \cite{Han-Hong-book}): if $\br:\Omega\to \R^3$ is an isometric immersion of class $C^2$, then the function $u=\mathbf{r}\cdot \mathbf{e}$ solves equation \eqref{eq-darb} for each $\mathbf{e}\in\S^2$, a fact which follows a direct computation using the Gauss equation for $\br$.  Conversely, given a smooth solution $u$ of \eqref{eq-darb} with $|du|_g<1$, one can construct a smooth mapping $\phi:\Omega\to\mathbb{R}^2$ such that $r=(\phi,u)$ defines a smooth isometric immersion of $(\Omega,g)$ into $\mathbb{R}^3$. This correspondence establishes the Darboux equation as a fundamental tool for studying the existence and properties of smooth isometric immersions.

In general, the validity of higher-order equations relating intrinsic and extrinsic geometry typically leads to uniqueness (also called rigidity) properties of isometric immersions—as exemplified by classical results for the Weyl problem \cite{CohnVossen1927,Herg1943}—or to non-existence results such as those of Hilbert \cite{Hilbert} and Efimov \cite{Efimov} for isometric immersions of class $C^2$ and higher.

By contrast, the celebrated Nash–Kuiper theorem \cite{Nash1954, Kui1955} established the counterintuitive abundance and flexibility of isometric immersions of class $C^1$: any short immersion can be uniformly approximated by isometric $C^1$ immersions, leading to the existence of highly oscillatory and non-unique isometries. This contrast naturally leads to the question of whether there exists a critical H\"older threshold $\theta_0$ that distinguishes the flexible regime ($C^{1,\theta}$ with $\theta<\theta_0$) from the rigid regime ($\theta>\theta_0$).

In recent years, significant work has been devoted to understanding this regime. The Nash-Kuiper theorem has been extended to $C^{1,\theta}$ isometric immersions for various ranges of $\theta$ in the works \cite{CDS2012, DIS2018, DI2020, CaoSz2022, CHI2025,Borisov2004}. On the other hand, rigidity has been shown to persist at higher regularity: the solution of the Weyl problem remains rigid in the class $C^{1,\theta}$ for $\theta > 2/3$ \cite{CDS2012,Borisov1958,Pakzad}; see also \cite{DeLellisPakzad} for a related rigidity result in the flat case.  Analogous flexibility and rigidity phenomena have been shown to hold for very weak solutions to the Monge–Amp\`ere equation, see \cite{LewickaPakzad, CaoSz, CHI, MartaSystems,Martasystems3,IL25,InLew25}. Note that these very weak solutions are defined through the very weak Hessian (cf. \cite{Iwaniec}) and are therefore different from our Definition \ref{de-weak-darboux} below, even though the Darboux equation is of Monge–Amp\`ere type.

In light of this flexibility-rigidity dichotomy, a natural and important question arises: does a suitable weak notion of higher-order constraint equations relating intrinsic and extrinsic geometry continue to hold in the low-regularity regime $C^{1,\theta}$ for sufficiently large $\theta$? Indeed, such weak constraint equations can serve as the key technical tools for establishing rigidity results. For instance, the central ingredient in the proof of the rigidity theorem in the class $C^{1,\theta}$ for $\theta > 2/3$ given in \cite{CDS2012} is the validity of an integrated version of the Theorema Egregium, $K_g  dV_g = N^*(dV_{\S^2})$, in this regularity regime. Similarly, in \cite{DI2020, CaoIn24}, a weak variant of the Gauss equation was shown to hold for isometries of class $C^{1,\theta}$ with $\theta > 1/2$, yielding certain rigidity properties for isometries with ``fixed boundary values".

Given this context and the fundamental connection between \eqref{eq-darb} and smooth isometric immersions, the goal of this paper is to extend the correspondence between isometric immersions and solutions to the Darboux equation to the low-regularity regime $C^{1,\theta}$. In this setting, the Riemannian Hessian $\nabla^2_g u$ is not classically defined, so we must introduce a suitable notion of weak solutions to \eqref{eq-darb}. It turns out, see Proposition \ref{pr:weak-darboux} below, that this is possible when $\theta>1/2$. 
\begin{definition}[Weak solutions to the Darboux equation]\label{de-weak-darboux}
    Let $\Omega\subset\R^2$ be a simply connected, bounded domain, and let  $g\in C^{2}(\bar \Omega; \R^{2\times2}_\sym)$ be a Riemannian metric with Gaussian curvature $K_g$. A function $u\in C^{1,\theta}( \bar \Omega)$ with $\theta>1/2$  is a weak solution of the Darboux equation \eqref{eq-darb} if $|du|_g<1$  and 
    \begin{equation}\label{eq-weak-darb}
    \frac{\det(\nabla^2_g u )}{(1-|du|_g^2)^{3/2}} dV_g = \frac{K_g}{(1-|du|_g^2)^{1/2}} dV_g
    \end{equation}
    as distributions on $\Omega$. 
\end{definition}
As we show in Proposition \ref{pr:weak-darboux}, this definition is well-posed. The main result of our paper is the following equivalence between $C^{1,\theta}$ isometric immersions and $C^{1,\theta}$ weak solutions to the Darboux equation.
\begin{theorem}\label{th-main}
    Let  $\theta>1/2$ be a H\"older exponent and let $g\in C^2( \bar{\Omega}; \R^{2\times 2}_\sym)$ be a metric on a simply connected and bounded domain $\Omega\subset\R^2$. Then the following two statements hold.
\begin{itemize}
    \item[(1)] If $\br\in C^{1,\theta}( \bar{\Omega};\R^3)$ is an isometric immersion of $( {\Omega}, g)$ into $\R^3$ and $u=\br\cdot \mathbf{e}$ satisfies $|du|_g<1$ for some unit vector $\mathbf{e}\in\S^2$, then $u\in C^{1,\theta}( \bar{\Omega})$ is a weak solution to the Darboux equation \eqref{eq-darb}.
    \item[(2)] If $u\in C^{1,\theta}( \bar{\Omega})$ is a weak solution to the Darboux equation \eqref{eq-darb}, then there exists $\phi\in C^{1, \theta}( \bar{\Omega};\R^2)$ such that the mapping $\br=(\phi, u): \Omega\to\R^3$ is an isometric immersion of $( \Omega, g)$ into $\R^3$.
\end{itemize}
\end{theorem}

The central ingredient in the proof of Theorem~\ref{th-main} is a generalization of the classical flatness criterion for Riemannian surfaces to the low-regularity regime. Classically, a two-dimensional Riemannian manifold $(\Omega,h)$ with smooth metric $h$ is flat (i.e., locally isometric to $(\mathbb{R}^2,e)$ where $e$ is the Euclidean metric) if and only if its Gaussian curvature $K_h$ vanishes identically. In Section~\ref{s:flatness}, we extend this criterion by introducing a weak notion of Gaussian curvature for Hölder continuous metrics $h\in C^{0,\theta}$ with $\theta>1/2$, using conformal coordinates. We then prove that the flatness characterization remains valid in this setting.

With this extension in place, the rest of the proof follows naturally by considering the auxiliary metric $h=g-du^2$, built from the original metric $g$ and the function $u$. In case (1), when $u=\mathbf{r}\cdot \mathbf{e}$ arises from a $C^{1,\theta}$ isometric immersion $\mathbf{r}$, the flatness criterion yields $K_h=0$. An approximation argument together with the classical curvature formula for smooth perturbations $g-dv^2$ then implies that $u$ satisfies \eqref{eq-darb}.
In case (2), where $u$ is a weak solution to the Darboux equation, an approximation  again yields $K_h=0$. The flatness criterion then provides the construction of the  the desired $C^{1,\theta}$ isometric immersion.

The remainder of the paper is organized as follows. Section \ref{s:prelims} introduces the notation used throughout the work. In Section \ref{se-proposition}, we justify the definition of weak solutions by showing that the distribution in \eqref{eq-weak-darb} is well-defined; this relies on a known result concerning distributional products, whose proof is provided in the appendix for completeness. Section \ref{s:conformal} recalls the construction and key properties of conformal coordinates, which play a central role in our analysis. Building on this, Section \ref{s:flatness} develops a weak notion of Gaussian curvature and establishes the extension of the flatness criterion to the low-regularity setting. Finally, Section \ref{se-proof-thm} contains the proof of the main theorem.

\section{Preliminaries}\label{s:prelims}
\subsection{Riemannian Hessian}\label{s:hessian}

Let $g$ be a Riemannian metric on a domain $\Omega \subset \mathbb{R}^2$, and let $u:\Omega \to \mathbb{R}$ be a $C^2$ function. The squared length of the differential $du$ with respect to $g$ is given in local coordinates by
\[
|du|_g^2 = g^{ij}\,\partial_i u \,\partial_j u\,,
\]
where $(g^{ij}) = g^{-1}$ denotes the inverse matrix of $(g_{ij})$, and we use the Einstein summation convention.

The Riemannian Hessian of $u$ is defined as the $(1,1)$-tensor
\[
\nabla_g^2 u : T\Omega \;\to\; T\Omega, 
\qquad 
X \;\mapsto\; \nabla_X \nabla_g u,
\]
where $\nabla_g u$ is the Riemannian gradient of $u$. In coordinates, one has
 \begin{equation}\label{e:hessian}
     \det \nabla_g^2 u = \frac{\det\left(\partial^2_{ij}u-\Gamma_{ij}^k\partial_k u \right)}{\det(g)}\,.
 \end{equation}
with $\Gamma^k_{ij}$ denoting the Christoffel symbols of $g$.

Note that the coefficients
$\partial^2_{ij}u-\Gamma_{ij}^k\partial_k u= \nabla^2 u(\partial_i,\partial_j)$
are precisely those of the covariant Hessian $\nabla^2 u = \nabla du$, a symmetric $(0,2)$-tensor. While some authors (see e.g. \cite{Han-Hong-book}) formulate the Darboux equation \eqref{eq-darb} in terms of $\nabla du$ rather than $\nabla_g^2 u$, we choose the Riemannian Hessian due to the fact that the determinant of a $(1,1)$-tensor is invariantly defined. 

\subsection{A formula for Gaussian curvature}
We also introduce a lemma in \cite{Han-Hong-book} for calculating the Gaussian curvature of perturbed metrics. 
\begin{lemma}[\cite{Han-Hong-book}]\label{le-g-du-k}
Let $(\mathcal{M}, g)$ be a smooth 2-dimensional Riemannian manifold and $u$  a smooth function on $\mathcal{M}$ with $|du|_g<1$. Then $h=g-du^2$ is a smooth Riemannian metric on $\mathcal{M}$ and the Gaussian curvature of $h$ is given by
    \[K_h=\frac{1}{1-|du|_g^2}\left\{K_g-\frac{\det(\nabla du)}{(1-|du|_g^2)\det g}\right\}\,.\]
\end{lemma}
The proof can be found in \cite[Lemma 2.1.2]{Han-Hong-book}.

\subsection{H\"older spaces} For an open set $\Omega\subset \R^2$, $l\in \N_0$ and $N\in \N$ we denote by $C^l(\bar \Omega;\R^N)$ the set of maps $f:\Omega\to \R^N$ whose derivatives $D^\gamma f$ are uniformly continuous on bounded subsets of $\Omega$ for all $|\gamma|\leq l$. The respective supremum norms are defined by
\begin{equation*}
\|f\|_0 = \sup_{x\in\Omega}|f(x)|\,, [f]_j=\max_{|\gamma|=j}\|D^\gamma f(x)\|_0, ~~\|f\|_l=\sum_{j=0}^l[f]_j,
\end{equation*}
For any $ 0<\theta\leq1$, the H\"older space $C^{l,\theta}(\bar \Omega,\R^N)$ is the Banach space of maps $f\in C^l(\bar \Omega;\R^N)$ with finite the H\"older norm $\|f\|_{C^{l, \theta}}$. Here, 
\[
\|f\|_{C^{l, \theta}}=\|f\|_{l}+[f]_{l,\theta}\,, \text{ where }
[f]_{l,\theta}=\sup_{x\neq y, x, y\in\Omega}\frac{|D^lf(x)-D^lf(y)|}{|x-y|^\theta}\,.
\]
We denote by $C^{l,\theta}_{loc}(\Omega,\R^N)$ the space of maps $f:\Omega\to \R^N$ such that $f\in C^{l,\theta}(\bar V,\R^N)$ for all open $V\subset \subset \Omega$. 
Moreover, $C_c^{l, \theta}(\Omega)$ is the space of functions in $C^{l,\theta}_{loc}(\Omega)$ with compact support in $\Omega$ and $\big(C_c^{l, \theta}(\Omega)\big)^*$ denotes its dual space.

We say that a sequence $\{f_\epsilon\}$ converges in $C^{l,\theta}_{loc}$ to $f\in C^{l,\theta}_{loc}(\Omega,\R^N)$, if $f_\epsilon \to f$ in  $C^{l,\theta}(\bar V,\R^N)$ for all open $V\subset \subset \Omega$.
\subsection{Mollification} We will use the regularization of maps by mollification with a standard, radially symmetric smoothing kernel $\Phi_\epsilon$. For $f\in C^{l,\theta}_{loc}(\Omega,\R^N)$ we will abbreviate $f_\epsilon = f\ast \Phi_\epsilon$. The maps $f_\epsilon $ are only defined on the set $\Omega_\epsilon = \{x\in \Omega:\mathrm{dist}(x,\partial\Omega)>\epsilon\}$. However, for any open $V\subset\subset \Omega$ it holds $V\subset \Omega_\epsilon$ for small enough $\epsilon$, and $f_\epsilon \to f $ in $C^{l,\theta'}(\bar V,\R^N)$ for any $\theta'<\theta$ (see for example Lemma 1 in \cite{CDS2012} for a proof of this fact). Therefore, $f_\epsilon \to f $ in $C^{l,\theta'}_{loc}$. 

\section{Weak solutions to the Darboux equation}\label{se-proposition}
In this section we show that the notion of a weak solution according to Definition \ref{de-weak-darboux} is well-posed. More precisely, we have 

\begin{proposition}\label{pr:weak-darboux}
    Let $\Omega \subset \R^2 $ be a bounded domain and let $g\in C^2(\bar \Omega;\R^{2\times 2}_\mathrm{sym})$ be a metric on $\Omega$. If $\theta>1/2$ and $u\in C^{1,\theta}(\bar \Omega)$ is a map with $|du|_g<1$ on $\bar \Omega$, then 
    \[\frac{\det(\nabla^2_g u )}{(1-|du|_g^2)^{3/2}} dV_g \]
    is well-defined as an element in $\left( C^{1,\theta}_c(\Omega)\right)^*$ with the following estimate:
\begin{equation}\label{eq.estimate on weak-darboux}
    \norm{\frac{\det(\nabla_g^2 u)}{(1-|du|^2_g)^{3/2}}}_{(C_c^{1,\theta}(\Omega))^*} \leq C (1+ \norm{\nabla u}_{C^{0,\theta}})\norm{\nabla u}_{C^{0,\theta}}\,.
\end{equation}    
Moreover, if $\{u_\epsilon\}\subset C^2_{loc}(\Omega)$ are such that $u_\epsilon\to u $ in $C^{1,\theta}_{loc}(\Omega)$, then 
    \begin{equation}\label{e:hessianconvergence}
        \lim_{\epsilon\to  0 }\int_{\Omega } \varphi \frac{\det(\nabla^2_g u_\epsilon)}{(1-|du_\epsilon|_g^2)^{3/2}} dV_g = \frac{\det(\nabla^2_g u )}{(1-|du|_g^2)^{3/2}} dV_g[\varphi]
    \end{equation}
    for any $\varphi\in C^{1,\theta}_c(\Omega)$. 
\end{proposition}
Here and hereafter, $C$ denotes a geometric constant which may depend on the metric $g$ and the domain $\Omega$.
To prove the above proposition we need the following result about distributional products of H\"older continuous functions. Several versions of this result are known in the literature (see for example Proposition 3.6 in \cite{DePakzad-2024}). For completeness we provide a short proof in the appendix.

\begin{proposition}\label{c:distrproduct} 
    Let $\Omega \subset \R^n$ be an open  set and $\theta>1/2$. Then the multilinear operator 
\[ C^1(\bar\Omega)\times C^1(\bar\Omega)\times C^1_c(\Omega) \ni (f,w,\varphi)\mapsto \int_{\Omega} \partial_i f w \varphi\,dx \]
has a continuous extension to $C^{0,\theta}(\bar \Omega)\times C^{0,\theta}(\bar \Omega)\times C_{c}^{0,\theta}(\Omega)$, which we denote by $\int \partial_i f w \varphi$. Consequently, for $f,w\in C^{0,\theta}(\bar\Omega)$ the distribution $\partial_i f w \in \left (C^{0,\theta}_c(\Omega)\right)^*$ can be defined by setting
\[ \partial_i f w[\varphi ] = \int\partial_i f w\varphi\,.\]
It satisfies the estimate  \[\|\partial_if w \|_{(C^{0,\theta}_c)^*}\leq C [f]_{0,\theta}[w]_{0,\theta}\,.\]
Moreover, if $\{f_n\},\{w_n\}$ converge in $C^{0,\theta}_{loc}(\Omega)$ to $f $ and $w$ respectively, then $$\partial_i f_n w_n[\varphi] \to \partial_i f w [\varphi]$$ for any $\varphi\in C^{0,\theta}_c(\Omega)$.
\end{proposition}

To prove Proposition \ref{pr:weak-darboux}, we need the following lemma for defining weak hessian determinant via differential forms.

\begin{lemma}\label{lem.extension of forms} Let $\Omega \subset \R^2 $ be a bounded domain and $\theta>\frac12$.
    If $b(x,y)dy^1\wedge dy^2$ is a $2$-form on $\Omega\times \R^2$ such that
\begin{equation}\label{eq.assumptions on form}
    \norm{b(\cdot,y)}_{C^{0,\theta}(\bar\Omega)} + \norm{D_xD_yb(\cdot,y)}_{C^{0,\theta}(\bar\Omega)} \le M
\end{equation}
for some constant $M\geq 1$, then \[b(x, \nabla u) d\partial_1u \wedge d\partial_2 u = (x,\nabla u)^\sharp \left( b(x,y)dy^1\wedge dy^2\right) \in \left(C^{1,\theta}_c(\Omega)\right)^*\]
extends continuously from $u\in C^2(\overline{\Omega})$ to $u\in C^{1,\theta}(\overline{\Omega})$ with estimate 
\[\norm{b(x, \nabla u) d\partial_1u \wedge d\partial_2 u}_{(C_c^{1,\theta}(\Omega))^*} \leq C M (1+ \norm{\nabla u}_{C^{0,\theta}})\,.\]
\end{lemma}

\begin{proof}
    Fix a vectorfield $B$ on $\Omega \times \R^2$ such that $$\div_y(B(x,y))=b(x,y).$$ In particular it can be chosen such that it satisfies similar estimates to 
    \eqref{eq.assumptions on form}, i.e.,
       \begin{equation}
        \label{eq-B-estimate}
        \sum_{i,j\in\{0, 1\}}\norm{D_x^iD_y^jB(\cdot,y)}_{C^{0,\theta}(\bar\Omega)} \le CM.
    \end{equation}
    
    Consider the 1-form $$\beta = B\restrl (dy^1\wedge dy^2)= B^i \star_y dy^i $$ and note that 
    \begin{align*}
        d\beta &= \div_y(B)\,( dy^1 \wedge dy^2) + \partial_{x_j}B^i (dx^j \wedge \star_y dy^i)\\
        &=:b(x,y) (dy^1\wedge dy^2) + b_{ij}(x,y) (dx^i\wedge dy^j)
    \end{align*}
    where we defined $b_{ij}\in C^{0,\theta}$ by the previous equality where we sum over repeated indices. Hence, 
    \[b(x,y)(dy^1\wedge dy^2)= d\beta - b_{ij}(x,y) (dx^i\wedge dy^j)\,.\]
    Now we observe that given $\nabla u \in C^{0,\theta}(\bar \Omega)$ by \eqref{eq-B-estimate} we have 
    \begin{align*}
        \norm{B^i(x,\nabla u)}_{C^{0,\theta}} &\leq C \left( \norm{B^i}_{C^{0,\theta}}+\norm{D_y B^i}_{C^{0,\theta}}\right)(1+ \norm{\nabla u}_{C^{0,\theta}})\\
        &\leq C M (1+ \norm{\nabla u}_{C^{0,\theta}})\,.
    \end{align*}
    Similarly we have 
    \begin{align*}
        \norm{b_{ij}(x,\nabla u)}_{C^{0,\theta}} &\leq \left( \norm{\partial_{x_j}B}_{C^{0,\theta}}+\norm{D_y\partial_{x_j} B}_{C^{0,\theta}}\right)(1+ \norm{\nabla u}_{C^{0,\theta}})\\
        &\leq C M (1+ \norm{\nabla u}_{C^{0,\theta}})\,.
    \end{align*}
    Hence in both cases we can appeal to Proposition \ref{c:distrproduct} to observe that 
    \begin{align*}
        \norm{B^i(x,\nabla u) d\partial_j u}_{(C^{0,\theta}_c(\Omega))^*}&\leq C M (1+ \norm{\nabla u}_{C^{0,\theta}})\norm{\nabla u}_{C^{0,\theta}}\,,\\
        \norm{b_{ij}(x,\nabla u) d\partial_j u}_{(C^{0,\theta}_c(\Omega))^*}&\leq C M (1+ \norm{\nabla u}_{C^{0,\theta}})\norm{\nabla u}_{C^{0,\theta}}\,.
    \end{align*}
    The continuous dependence on $u$ follows from the continuous dependence observed in Proposition \ref{c:distrproduct} and the observation that the maps
    \[ \nabla u \mapsto b_{ij}(\cdot,\nabla u), \quad \nabla u \mapsto B^i(\cdot,\nabla u)\]
    are continuous maps from $C^{0,\theta}(\bar\Omega)$ into $C^{0,\theta}(\bar \Omega)$.
\end{proof}
\begin{proof}[Proof of Proposition \ref{pr:weak-darboux}]
    As recalled in Section \ref{s:hessian} \[ \det(\nabla^2_g u)  = \frac{\det\left(\partial_{ij}^2u - \Gamma^{k}_{ij} \partial_k u\right)}{\det(g)}\] 
    and thus
    \begin{align*}
        \frac{\det(\nabla^2_g u)}{(1-|du|_g^2)^{3/2}}
        =\frac{\det(D^2u)}{\det(g) (1-|du|_g^2)^{3/2}} - \frac{\operatorname{cof}(\Gamma^k\partial_ku)\colon D^2u}{\det(g)(1-|du|_g^2)^{3/2}} + \frac{\det(\Gamma^k\partial_k u)}{\det(g)(1-|du|_g^2)^{3/2}}\,.
    \end{align*}
    
   Note that by assumption there is $\epsilon>0$ such that $g^{ij} \partial_iu \partial_j u< (1-2\epsilon)^2$ on $\overline{\Omega}$. Now fix any $\tilde{b}(t)\in C^\infty(\R_+)$ such that $$\tilde{b}(t)=\frac{1}{(1-t)^{3/2}} \text{ on } |t|<1-\epsilon$$ and define $b, A, E:\Omega\times \R^2 \to \R$ by 
   \begin{align*}
         b(x,y)&= \det(g)^{-1}\tilde{b}(g^{ij}(x)y_iy_j)\,,\\
         A(x,y) &= b(x,y) \operatorname{cof}( \Gamma^k(x)y_k)\,,\\
         E(x,y) &= b(x,y) \det(\Gamma^k(x)y_k)\,.
   \end{align*}
    With these definitions it follows that for every $u \in C^2(\bar{\Omega})$ satisfying $|du|_g< 1-\epsilon$ on $\Omega$ it holds 
    \[\frac{\det(\nabla^2_g u)}{(1-|du|_g^2)^{3/2}} = b(x,\nabla u) d\partial_1u\wedge d\partial_2u + A(x,\nabla u)\colon D^2u + E(x,\nabla u)\,.\]
    
    The first term extends locally continuously to $u\in C^{1,\theta}(\bar \Omega)$ by Lemma \ref{lem.extension of forms}. For the second and third term we observe that $\nabla u \mapsto A(\cdot,\nabla u),\,\, \nabla u \mapsto E(\cdot, \nabla u)$ are continuously mapping $C^{0,\theta}(\bar \Omega)$ into $C^{0,\theta}(\bar \Omega)$. Hence we can apply Proposition \ref{c:distrproduct} to the second term, while the third is classically defined. Finally the conclusion follows.
\end{proof}

\section{Distributional Gaussian curvature and flatness}\label{s:flatness}
    In this section we define a distributional Gaussian curvature for H\"older continuous metrics $h\in C^{0,\theta}$ (see also \cite{Pakzad} for a distributional Gaussian curvature for $h\in C^1$). We then show that the classical characterization of flatness extends to this low-regularity regime: a two-dimensional Riemannian manifold $(\Omega,h)$ with smooth metric is flat (i.e., locally isometric to $\R^2$) if and only if its Gaussian curvature $K_h$ vanishes identically.
    
    We will exploit the two-dimensional setting by using isothermal (or conformal) coordinates. We review their definition and key properties in the next subsection. 
\subsection{Conformal charts}\label{s:conformal}
 \begin{definition}\label{de-conformal}
        Let $\Omega\subset \C $ be an open set and $h\in C^{0,\theta}_{loc}( \Omega, \R^{2\times 2}_\sym)$ a metric. We call $(\psi,  V)$ a \emph{conformal chart} for $h$ if $ V\subset \Omega$ is open, $\psi:V\to\psi(V) \subset \C $ is an orientation preserving diffeomorphism, and 
        \[ h = \psi^\sharp (e^{2\sigma}|dz|^2) \text{ on } V.\]
      Here, $\sigma$ is determined by $\psi $ and $h$ through  $\sigma =\frac12 \log \left(\frac{ \mathrm{tr} h}{|D\psi|^2}\circ \psi^{-1}\right)$. Nevertheless, for notational reasons we will denote the chart by $(\psi,\sigma,V)$.
    \end{definition}
    The map $\psi$ yields coordinates in which the metric $h$ is diagonal. They are also known as isothermal coordinates. 
\begin{remark}\label{r:conformal}

   \begin{enumerate}
   \item The existence of conformal charts for a given metric is classical (see \cite{Ahlfors} and references therein), and can be reduced to solving the Beltrami equation $\psi_{\bar z} = \mu \psi_z$ on $V$. Conversely,  every conformal chart arises as a solution to the Beltrami equation. The coefficient $\mu$, which satisfies $|\mu|<1$ on $\Omega$, is determined by the metric $h$. In particular, if $h\in C^{l,\theta}_{loc}$, then $\mu \in C^{l,\theta}_{loc}$ as well.
   \item If $h\in C^{l,\theta}_{loc}(\Omega, \R^{2\times 2}_\sym)$ is a metric with $l\in \N_0, \theta\in (0,1]$, then any chart $(\psi, \sigma, V)$ satisfies  $\psi \in C^{l+1,\theta}_{loc}(V), \sigma \in C^{l,\theta}_{loc}(\psi(V))$ (cf. Theorem 15.6.2 in \cite{AIM2009}).
   \end{enumerate}
\end{remark}
We will need the following facts. 
\begin{lemma}\label{l:charts}
    Assume $\Omega \subset \C$ is an open, bounded set and $h\in C^{0,\theta}(\bar \Omega,\R^{2\times 2}_\sym)$. Then there exists a ``global'' conformal chart  $(\psi,\sigma,\Omega)$, satisfying in addition $\psi\in C^{1,\theta}(\bar \Omega)$. Moreover, if $\{h_\epsilon\}\subset  C^{0,\theta}(\bar \Omega,\R^{2\times 2}_\sym)$ is a sequence of metrics such that $h_\epsilon\to h$ in $C^{0,\theta}$ there are ``global''  conformal charts $(\psi_\epsilon,\sigma_\epsilon, \Omega)$ for $h_\epsilon$ satisfying
    \begin{equation}\label{e:convofcharts}
    \|\psi_\epsilon-\psi\|_{C^{1,\theta}(\bar \Omega)} + \|\sigma_\epsilon-\sigma\|_{C^{0, \theta}(\bar \Omega)} \to 0\,.
\end{equation} 
\end{lemma}
For the reader's convenience we provide a proof of the preceding lemma in the appendix.

\subsection{Distributional Gaussian curvature} We now define a notion of Gaussian curvature for H\"older continuous metrics.  Recall that for a smooth metric $h$, the Gaussian curvature $K_h$ in conformal coordinates $(\psi,\sigma,V)$ is given by $$K_h\circ\psi^{-1} = -2e^{-2\sigma} \Delta \sigma.$$
Since moreover $\det h\circ \psi^{-1} = e^{4\sigma}$ it follows that 
    \[ \int_V \varphi K_h dV_h = \int_{\psi(V)}(\varphi\circ \psi^{-1})(K_h\circ\psi^{-1})(\sqrt{\det h }\circ \psi^{-1})dz =- 2\int_{\psi(V)} (\varphi\circ \psi^{-1}) \Delta \sigma dz\,.  \]
    This motivates the following 
\begin{definition} Let $\Omega\subset \R^2 $ be open, $\theta>1/2$, and let $h\in C^{0, \theta}_{loc}( {\Omega}; \R^{2\times2}_\sym)$ be a metric. We define the distributional Gaussian curvature by
    \begin{equation}\label{eq-distri-gauss}
K_hdV_h[\varphi]=2\int\nabla(\varphi\circ\psi^{-1})\cdot\nabla\sigma 
    \end{equation}
    for $\varphi\in C^\infty_c(\Omega)$, where $(\psi,\sigma,V)$ is any conformal chart such that $\mathrm{spt}(\varphi)\subset V$. 
\end{definition}

\begin{proposition}\label{pr-kdv-def}
    The distributional Gaussian curvature is well-defined by \eqref{eq-distri-gauss} and extends to $K_hdV_h\in \left(C^{1,\theta}_c(\Omega)\right)^*$. Moreover, if $h_\epsilon \to h $ in $C^{0,\theta}_{loc}(\Omega)$ then $K_{h_\epsilon}dV_{h_\epsilon} \to K_hdV_h$ as distributions. 
\end{proposition}

\begin{proof} Let $\varphi\in C^\infty_c(\Omega)$ and $(\psi,\sigma, V)$ a conformal chart such that $\mathrm{spt}(\varphi)\subset V$. It follows from Remark \ref{r:conformal} (2) that $\psi \in C^{1,\theta}_{loc}(V)$ and $\sigma \in C^{0,\theta}_{loc}(\psi(V))$. As $\nabla (\varphi\circ \psi^{-1}) \in C^{0, \theta}_c(V)$, the right-hand side of \eqref{eq-distri-gauss} is well-defined as a distributional product thanks to Proposition \ref{c:distrproduct}.

Now assume that $(\psi_1,\sigma_1,V_1)$, $(\psi_2,\sigma_2,V_2)$ are two conformal charts with $\mathrm{spt}(\varphi)\subset V_i$, i.e., $\mathrm{spt}(\varphi)\subset V_1\cap V_2 =: \Omega_0$. We claim that 
\begin{equation*}
        \int\nabla(\varphi\circ\psi_{1}^{-1})\cdot\nabla\sigma_{1}
        =\int\nabla(\varphi\circ\psi_{2}^{-1})\cdot\nabla\sigma_{2}\,.
    \end{equation*}
    For this we consider mapping
    \[f:=\psi_1\circ\psi_2^{-1}: ~~{\psi_2(\Omega_0)}\to{\psi_1(\Omega_0)}.\]
Observe that by definition  
\[ \psi_1^\sharp(e^{2\sigma_1}|dz|^2) = h = \psi_2^\sharp (e^{2\sigma_2}|dz|^2)\]
on $\Omega_0$, so that 
    \begin{align*}
      e^{2\sigma_2}|dz|^2= (\psi_2^{-1})^{\sharp}h=e^{2(\sigma_1\circ f)}f^\sharp |dz|^2\,.
    \end{align*}
Consequently, $f$ is conformal and hence, since it is also orientation preserving, holomorphic. In particular, $f^\sharp |dz|^2 = |f'|^2 |dz|^2$,  
    which implies
    \begin{equation}
        \label{eq-sigma12}
        \sigma_2=\sigma_1\circ f+\ln|f'|
    \end{equation}
    on $\psi_2(\Omega_0)$. Note that $\ln|f'|$ is harmonic, so that     \begin{equation}
        \label{eq-f-harmonic}
        \int_{\psi_2(\Omega_0)}\nabla\ln|f'|\cdot\nabla\Psi dz=0, \text{ for any }\Psi\in C_c^1(\psi_2(\Omega_0)).
    \end{equation}

    Now consider the mollification $(\sigma_2)_{\epsilon}$ of $\sigma_2$ with a standard, radially symmetric mollifier. By \eqref{eq-sigma12} it holds 
    \[ (\sigma_2)_\epsilon = (\sigma_1\circ f)_\epsilon+(\ln|f'|)_\epsilon =(\sigma_1\circ f)_\epsilon+\ln|f'|\,, \]
    since $\ln|f'|$ is harmonic. 
    Since $(\sigma_2)_\epsilon\to \sigma_2$ in $C^{0,\theta'}_{loc}(\psi_2(V_2))$ for any $1/2<\theta'<\theta$ and $\psi_2^{-1}=\psi_1^{-1}\circ f $ we have by continuity of the distributional product and \eqref{eq-f-harmonic}
    \begin{align*}
        \int\nabla(\varphi\circ\psi_{2}^{-1})\cdot\nabla\sigma_{2} &= \lim_{\epsilon\to 0} \int_{\psi_2(\Omega_0)}\nabla(\varphi\circ\psi_{2}^{-1})\cdot\nabla(\sigma_{2})_\epsilon dz \\
        &=\lim_{\epsilon\to 0} \int_{\psi_2(\Omega_0)}\nabla(\varphi\circ\psi_{2}^{-1})\cdot\nabla\left((\sigma_1\circ f)_\epsilon+\ln|f'|\right)dz \\
        &=\lim_{\epsilon\to 0}\int_{\psi_2(\Omega_0)}\nabla(\varphi\circ\psi_{1}^{-1}\circ f)\cdot\nabla\left((\sigma_1)_\epsilon\circ f\right)dz  \\
     &\quad + \lim_{\epsilon\to 0}\int_{\psi_2(\Omega_0)}\nabla(\varphi\circ\psi_{1}^{-1}\circ f)\cdot\nabla\left((\sigma_1\circ f)_\epsilon-(\sigma_1)_\epsilon\circ f\right)dz\\
      &=: I+I\!\!I\,.
    \end{align*}
    By the conformal invariance of the Dirichlet integral we get 
     \[ I = \lim_{\epsilon\to 0}\int_{\psi_1(\Omega_0)}\nabla(\varphi\circ\psi_{1}^{-1})\cdot\nabla(\sigma_1)_\epsilon dz=\int \nabla(\varphi\circ\psi_{1}^{-1})\cdot\nabla\sigma_1 \,. \]
     On the other hand, since 
     \[(\sigma_1\circ f)_\epsilon-(\sigma_1)_\epsilon\circ f=((\sigma_1\circ f)_\epsilon-\sigma_1\circ f) + (\sigma_1\circ f-(\sigma_1)_\epsilon\circ f )\to 0 \] 
     in $C^{0,\theta'}$ for any $1/2<\theta'<\theta$ we get by the continuity of the distributional product that $I\!\!I =0$, proving the claim.
     
For the second assertion, let  $h_\epsilon$ be a sequence of metrics such that $h_\epsilon \to h $  in $C^{0,\theta}_{loc}(\Omega)$ and let $\varphi\in C^{\infty}_c(\Omega)$. Given an open $V\subset\subset \Omega $ such that $\mathrm{spt}\,\varphi \subset V$ we can find global conformal charts $(\psi, \sigma,V)$, $(\psi_\epsilon, \sigma_\epsilon,V)$ for $h$ and $h_\epsilon$ such that \eqref{e:convofcharts} holds with $\Omega $ replaced by $V$. Consequently, from the continuity of the distributional product in Proposition \ref{c:distrproduct} and the convergence \eqref{e:convofcharts} we find
\[ K_hdV_h[\varphi]=2\int \nabla (\varphi\circ \psi^{-1})\cdot\nabla \sigma dz = 2\lim_{\epsilon \to 0} \int \nabla (\varphi\circ \psi_\epsilon^{-1})\cdot\nabla \sigma_\epsilon dz= \lim_{\epsilon\to 0}K_{h_\epsilon}dV_{h_\epsilon}[\varphi] \,.\qedhere \]
\end{proof}

\subsection{Flatness}
We now show the generalization to metrics with low regularity of a classical theorem concerning the flatness of Riemannian manifolds. 

\begin{proposition}\label{cor-kdv-0}
     Let $\theta>1/2$ and $h\in C^{0, \theta}( \bar{\Omega}; \R^{2\times2}_\sym)$ be a metric on a simply connected and bounded domain $\Omega\subset\R^2$. Then there exists an isometric immersion $\phi\in C^{1,\theta}(\bar \Omega,\R^2)$ of $(\Omega,h)$ into $\R^2$ if and only if $K_hdV_h[\varphi]=0$ for  any $\varphi\in C^\infty_{c}(\Omega).$ 
\end{proposition}

\begin{proof}
     Let $\phi\in C^{1,\theta}(\bar \Omega,\R^2)$ be an isometric immersion from $( \Omega, h)$ into $\R^2$. As an immersion $\phi$ is a local diffeomorphism, so there exists an open cover $\{V_i\}$ of $\Omega$ such that $\phi\vert_{V_i}$ is a diffeomorphism onto its image for every $i$. Since 
    $h=\phi^\sharp|dz|^2$, $(\phi\vert_{V_i},V_i)$ is a conformal chart (upto composing it with a reflection) with coefficient $\sigma_i = 0$ for every $i$. For a given $\varphi\in C^{\infty}_c(\Omega)$ we can find a finite subcover $V_1,\ldots,V_N$ of $\mathrm{spt }\, \varphi$, and choose  a partition of unity $\{\chi_i\}$ subordinate to the cover. Since the right-hand side in \eqref{eq-distri-gauss} is independent of the choice of conformal chart it follows 
    
    \[ K_hdV_h[\varphi] = \sum_{i=1}^N K_hdV_h[\chi_i\varphi]=2\sum_{i=1}^N \int \nabla(\chi_i\varphi\circ \phi\vert_{V_i}^{-1})\cdot \nabla \sigma_i = 0\,.\]

    Now assume conversely that $K_hdV_h = 0$ as a distribution and let $(\psi,\sigma,\Omega)$ be a global conformal chart for $h$. We claim that $\sigma$ is harmonic on $\psi(\Omega)$. Indeed, fix $\varphi \in C^\infty_c(\psi(\Omega))$, and let $\sigma_\epsilon$ be smooth approximations by mollification in $C^{0,\theta}_{loc}(\psi(\Omega))$ of $\sigma$. Since $\tilde \varphi = \varphi\circ \psi \in C^{1,\theta}_c(\Omega)$ is a valid testfunction in $K_hdV_h$ it follows by the continuity of the distributional product 
    \[ \int_{\psi(\Omega)} \sigma \Delta \varphi dz = -\lim_{\epsilon\to 0} \int_{\psi(\Omega)} \nabla \sigma_\epsilon \cdot \nabla \varphi dz =-\int_{\psi(\Omega)}  \nabla \sigma \cdot \nabla (\tilde \varphi\circ \psi^{-1}) = -\frac12 K_hdV_h[\tilde \varphi ] = 0\,. \]
Hence, by Weyl's lemma for the Laplace equation it follows that $\sigma$ is smooth and harmonic. 
Since $\Omega$ and so $\psi(\Omega)$ are simply connected,  $\sigma$ is the real part of some holomorphic function $f$, i.e., $\sigma=\Re f$. Appealing once more to the simply connectedness of $\psi(\Omega)$ there is a holomorphic function $F\colon\psi(\Omega)\to\C$ which is a primitive of $e^f$, i.e., $F'=e^f$. It satisfies $$|DF|^2= e^{2\Re(f)}= e^{2\sigma}.$$
Consequently, the mapping
\[\phi:=F\circ \psi:\Omega\to\C\]
is an isometric immersion from $(\Omega, h)$ to $\R^2$. 
\end{proof}

\section{Proof of Theorem \ref{th-main}}\label{se-proof-thm}
The proof of our main theorem now follows from Proposition \ref{pr:weak-darboux} and Proposition \ref{cor-kdv-0}. Since the statements are in two parts, we divide the proof into two subsections.

\subsection{Proof of Theorem \ref{th-main} (1)} Fix a test-function $\varphi\in C^{1,\theta}_c(\Omega)$ and let $V\subset \subset \Omega$ be an open set such that $\mathrm{spt} \,\varphi \subset V$. Consider then the mollification $\br_\epsilon$ of the immersion $r$. For $\epsilon$ small enough,  $u_\epsilon=\br_\epsilon\cdot \mathbf{e}$ is well-defined and smooth on $V$ with  $|du_\epsilon|_g<1$ on $\bar V$. Moreover, 
\begin{equation}
    \label{eq-u-appro}
    u_\epsilon\to u \text{ in } C^{0, \theta'}_{loc}( \Omega),
\end{equation}
for any $\theta'\in(1/2, \theta).$   Define 
\[h_\epsilon=g-du_\epsilon^2 \text{ and } h=g-du^2.\]
Then 
\begin{equation}
    \label{eq-heps-conv}
    h_\epsilon\to h \text{ in } C^{0, \theta'}_{loc}( \Omega) \text{ and }  h\in C^{0, \theta'}( \bar \Omega)
\end{equation}
for any $\theta'\in(1/2, \theta).$ With the help of Lemma \ref{le-g-du-k}, we can compute the Gaussian curvature for the metric $h_\epsilon$ as 
\begin{align*}
    K_{h_\epsilon}=\frac{1}{1-|du_\epsilon|_g^2}\left\{K_g-\frac{\det(\nabla du_\epsilon)}{(1-|du_\epsilon|_g^2)\det g}\right\}\,.
\end{align*}
A direct calculation also implies
\[\det h_\epsilon=(1-|du_\epsilon|_g^2)\det g\,.\]
Hence we further have
\begin{equation}
    \label{eq-kdvepsilon}
    K_{h_\epsilon}dV_{h_\epsilon}=K_{h_\epsilon}\sqrt{\det h_\epsilon}dx
    =\left\{\frac{K_g\sqrt{\det g}}{(1-|du_\epsilon|_g^2)^{1/2}}-\frac{\det(\nabla du_\epsilon)}{(1-|du_\epsilon|_g^2)^{3/2}\sqrt{\det g}}\right\}dx.
\end{equation}
With the aid of the convergence of $h_\epsilon$ and $u_\epsilon$, by Proposition \ref{pr-kdv-def}, we have
\begin{equation}
    \label{eq-kdv-con}
    K_{h_\epsilon}dV_{h_\epsilon}[\varphi]\to K_hdV_h[\varphi]\,.
\end{equation}
Since $\br$ is an isometric immersion of $( \Omega, g)$ into $\R^3$ and $|du|_g<1$ in $ \Omega$, the metric $h=g-du^2$ is isometric to $\R^2$, which by Proposition \ref{cor-kdv-0} gives 
\begin{equation}
    \label{eq-kdv0}
     K_hdV_h[\varphi]=0.
\end{equation}
Thus from \eqref{eq-kdvepsilon}, \eqref{eq-kdv-con} and \eqref{eq-kdv0}, we have
\begin{equation}\label{eq-ueps-darb}
    \left(\int_\Omega\frac{\varphi K_g\sqrt{\det g}}{(1-|du_\epsilon|_g^2)^{1/2}}dx-\int_\Omega\frac{\varphi\det(\nabla du_\epsilon)}{(1-|du_\epsilon|_g^2)^{3/2}\sqrt{\det g}}dx\right)\to0.
\end{equation} Combining this with the convergences \eqref{eq-u-appro} and \eqref{e:hessianconvergence}, and the arbitrariness of $\varphi$ yields \eqref{eq-weak-darb} for $u$.

\subsection{Proof of Theorem \ref{th-main} (2)} Consider the metric $h=g-du^2 \in C^{0,\theta}(\bar \Omega;\R^{2\times2}_\sym)$. We claim that the distributional Gaussian curvature vanishes identically. Proposition \ref{cor-kdv-0} then guarantees the existence of an isometric immersion $\phi\in C^{1,\theta}(\bar \Omega,\R^2)$ of $(\Omega,h)$ into $\R^2$. Setting $\br=(\phi, u)$ then yields the desired isomeric immersion of $(\Omega, g)$ into $\R^3$, since 
\[ |d\br|^2 = |d\phi|^2+du^2 = h+du^2 =g\,.\] 

It therefore remains to show that $K_hdV_h[\varphi] =0$. Fix again $\varphi\in C^{1,\theta}_c(\Omega) $ and an open set $V\subset \subset \Omega$ with $\mathrm{spt}\, \varphi\subset V$. Mollifying $u$ to get a smooth function $u_\epsilon$, we also have a smooth metric $h_\epsilon=g-du_\epsilon^2$ on $V$ if $\epsilon$ is small enough. Since $u\in C^{1, \theta}( \bar \Omega),$ we immediately get
\begin{equation}\label{eq-approx-uh}
    u_\epsilon\to u \text{ in }C^{1, \theta'}_{loc}(\Omega), \quad h_\epsilon\to h\text{ in } C^{0, \theta'}_{loc}(\Omega),
\end{equation}
for any $1/2<\theta'<\theta.$ 
It follows from Proposition \ref{pr-kdv-def} and \eqref{eq-kdvepsilon} that
\begin{align*}
       K_hdV_h[\varphi] &= \lim_{\epsilon\to 0} \int  \varphi K_{h_\epsilon}\sqrt{\det h_\epsilon}dx\\ 
       &= \lim_{\epsilon\to 0 }\int \varphi 
    \left\{\frac{K_g\sqrt{\det g}}{(1-|du_\epsilon|^2)^{1/2}}-\frac{\det(\nabla du_\epsilon)}{(1-|du_\epsilon|^2)^{3/2}\sqrt{\det g}}\right\}dx \\
    &=  \frac{K_g}{(1-|du|_g^2)^{1/2}} dV_g[\varphi]-\frac{\det(\nabla^2_g u )}{(1-|du|_g^2)^{3/2}} dV_g[\varphi]=0\,,
\end{align*}
where the last line follows as a consequence of \eqref{eq-approx-uh} in view of \eqref{e:hessianconvergence} and the fact that $u$ is a weak solution to the Darboux equation. Since $\varphi$ was arbitrary, this shows the claim and finishes the proof.

\begin{appendix}

\section{Distributional products}

The proof of Proposition \ref{c:distrproduct} follows as a corollary from the following two lemmas.
\begin{lemma}\label{lem.product}
    The trilinear operator
    \[ \mathcal{S}(\R^n)^3\ni (f,g,h)\mapsto  \int_{\R^n} Df(x) g(x) h(x)\, dx \]
    extends continuously to a trilinear operator on $B^{s}_{p_1,q}(\R^n)\times B^{1-s}_{p_2,q'}(\R^n)\times B^{1-s}_{p_3,q'}(\R^n)$ denoted by $\int Df g h$ for any set of parameters satisfying $\frac{1}{p_1}+ \frac{1}{p_2}+ \frac{1}{p_3}=1$, $\frac{1}{q}+\frac{1}{q'}=1$ and $0<s<1$.
\end{lemma}
Here, $B^{s}_{p,q}(\R^n)$ denotes the (inhomogenous) Besov-space. Even though the result is probably well known and perhaps not optimal, we will present a short proof for the sake of completeness. 

\begin{proof}
    Let us denote with $F(x,t)=\Phi_t *f(x) ,G(x,t)=\Phi_t *g(x),H(x,t)=\Phi_t *h(x)$ be the smooth extensions of $f(x),g(x),h(x)$ to $\R^{n+1}_+$ having the ``good trace'' property, i.e.,
    \begin{align*}
        \norm{t^{|\alpha|-s}\norm{\partial^\alpha F(\cdot, t)}_{L^p}}_{L^q(\R_+, \frac{dt}{t})}&\leq C [f]_{B^{s}_{p,q}} \text{ for } \alpha \in \N_0^{n+1}  \text{ with } |\alpha|>0\,,\\
       \sup_{t} \norm{F(\cdot ,t)}_{L^p} &\leq C [f]_{B^{s}_{p,q}}\,.
    \end{align*}
    The first estimate can for instance be found in \cite[Theorem 1.4]{MironescuRuss2015} and the second is a simple consequence of $F$ being a mollification of $f$. 
    Having these extensions at hand the proof is straight forward. Recall that $\partial_j f g h \, dx = gh d(f\star_{\R^n} dx^j)$ and properties of mollifier so that 
    \begin{align*}
        &\lvert \int \partial_j f g h \,dx\rvert = \lvert \int_{\R^n\times\{0\}} GHd(F\star_{\R^n} dx^j)\rvert = \lvert \int_0^\infty \int d (F \star_{\R^n} dx^j) \wedge d(GH) \rvert\\
         \le & \int_{0}^\infty t \norm{\partial_j F(\cdot, t)}_{L^{p_1}}\left( \norm{\partial_t G(\cdot, t)}_{L^{p_2}}\norm{H(\cdot,t)}_{L^{p_3}}+ \norm{ G(\cdot, t)}_{L^{p_2}}\norm{\partial_t H(\cdot,t)}_{L^{p_3}}\right) \,\frac{dt}{t}\\
        &\quad+\int_{0}^\infty t \norm{\partial_t F(\cdot, t)}_{L^{p_1}}\left( \norm{\partial_j G(\cdot, t)}_{L^{p_2}}\norm{H(\cdot,t)}_{L^{p_3}}+ \norm{ G(\cdot, t)}_{L^{p_2}}\norm{\partial_j H(\cdot,t)}_{L^{p_3}}\right) \,\frac{dt}{t}  \\
        \le& 2\left(\sup_{t}\norm{H(\cdot,t)}_{L^{p_3}}\right) \norm{t^{1-s}\norm{D F(\cdot, t)}_{L^{p_1}}}_{L^{q}(\R_+, \frac{dt}{t})} \norm{t^{1-(1-s)}\norm{D G(\cdot, t)}_{L^{p_2}}}_{L^{q'}(\R_+, \frac{dt}{t})}\\
        &\quad+2\left(\sup_{t}\norm{G(\cdot,t)}_{L^{p_2}}\right) \norm{t^{1-s}\norm{D F(\cdot, t)}_{L^{p_1}}}_{L^{q}(\R_+, \frac{dt}{t})} \norm{t^{1-(1-s)}\norm{D H(\cdot, t)}_{L^{p_3}}}_{L^{q'}(\R_+, \frac{dt}{t})}\\
        \leq & C [f]_{B^{s}_{p_1,q}}[g]_{B^{1-s}_{p_2,q'}}[h]_{B^{1-s}_{p_3,q'}}\,.
    \end{align*}
    Together with the tri-linearity of the product on the left the claim follows.  
\end{proof}

Recall  the following standard embedding (cf. \cite[Theorem 2.7.1]{Triebel} ) lemma.
\begin{lemma}\label{lem.embedding}
   $C^{0, \theta}(\R^n)$ embeds continuously into $B^{s}_{\infty, q}(\R^n)$ for any $s<\theta$ and $q\ge 1$.  
\end{lemma}


\begin{proof}[Proof of Proposition \ref{c:distrproduct}]
    This is a direct consequence of the previous Lemma \ref{lem.product} and Lemma \ref{lem.embedding} after using the classical Whitney extension theorem to obtain compactly supported extensions $\hat{f},\hat{g}$ with same norms bounds. 
\end{proof}

\section{Conformal charts: A proof of Lemma \ref{l:charts}}

    We first extend $h$ to $\tilde  h\in C^{0,\theta}(\C,\R^{2\times2}_\sym)$. By continuity, $\tilde h $ is a metric on $\bar\Omega'$ for some open, bounded $\Omega' \supset\supset \Omega $. Let $\mu_{\tilde h} $ denote the Beltrami coefficient for $\tilde h$, which satisfies $\mu_{\tilde h } \in C^{0,\theta}(\bar \Omega')$ and $|\mu_{\tilde h} |\leq k <1$ on $\Omega'$. Choose a cutoff function $\eta \in C^\infty_c(\Omega')$ such that $0\leq\eta \leq 1$ and $\eta = 1 $ on $\Omega$ and set $\mu = \eta\mu_{\tilde h}$, so that $\mu\in C^{0,\theta}_c(\C)$ with $|\mu|\leq k <1 $ everywhere. Now let $\psi $ be the principal solution to the Beltrami equation $\psi_{\bar z} = \mu \psi_z$ (cf. Theorem 5.1.2 in \cite{AIM2009}). It is a homeomorphism (by Theorem 5.3.2 of \cite{AIM2009}) and moreover $\psi \in C^{1,\theta}_{loc}(\C)$ by Theorem 15.6.2 of \cite{AIM2009}. Theorem 5 in \cite{Ahlfors} and $|\mu|\leq k<1$ then imply that the Jacobian $J\psi = |\psi_z|^2 - |\psi_{\bar z}|^2$ is everywhere positive, so that $\psi $ is an orientation preserving diffeomorphism. Since on $\Omega $ it holds $\mu = \mu_{\tilde h}$ it follows that $(\psi,\Omega)$ is a conformal chart. 
    
In the same way we find global conformal charts $(\psi_\epsilon,\sigma_\epsilon)$ for the metrics $h_\epsilon$. Let $f^\epsilon= \psi_\epsilon-\psi \in C^{1,\theta}_{loc}(\C)$ and observe 
    \[ f^\epsilon_{\bar z} = \mu f^\epsilon_{z} + (\mu-\mu_\epsilon)f^\epsilon_z.\]
    We would like to employ Theorem 15.0.6 in \cite{AIM2009} to the above inhomogeneous Beltrami equation to show $f^\epsilon\to0$ as $\epsilon\to0$. To do this we require the uniform boundedness of $f^\epsilon$ in $C^{1, \theta}_{loc}(\C)$ and $L^2(\C)$. Observe that for any  principal solution $\zeta$ to the Beltrami equation $\psi_{\bar z } = \mu \psi_z$ with $\mu\in C^{0, \theta}_c(\C) $, one has $$\|\zeta-z\|_{L^\infty(\C)}\leq C \|\mu\|_{L^\infty(\C)}$$ as a consequence of Theorem 4.3.11 in \cite{AIM2009} (where the constant $C$ depends on the support of $\mu$). Note that $\tilde\zeta= \zeta-z $ satisfies $D{\tilde\zeta}\in L^2(\C)$ and $$\tilde\zeta_{\bar z} = \mu \tilde\zeta_z + \mu. $$ 
    Hence by the Schauder estimates in Theorem 15.0.6 in \cite{AIM2009} we get 
    \[  [D\zeta]_{C^{0, \theta}(\C)} = [D{\tilde\zeta}]_{C^{0, \theta}(\C)}\leq C\left([\mu]_\theta + ([\mu]_\theta)^{1+1/\theta} \|\mu\|_\infty\right)\,.\]
     Combining with the $L^\infty$ bound and  interpolation, in our case we see that $\psi_\epsilon,\psi$ and so also $f^\epsilon$ are uniformly (in $\epsilon$) bounded in $C^{1,\theta}(K_0)$ for any compact $K_0\subset \C$. 

   Moreover, from the fact that $\psi_\epsilon $ and $\psi $ are principal solutions it follows that $Df^\epsilon\in L^2(\C)$. Consequently, the Schauder estimates in Theorem 15.0.6 in \cite{AIM2009} imply 
   \begin{align*} [Df^\epsilon]_{C^{0,\theta}(\C)} &\leq C\left( [(\mu-\mu_\epsilon)f^\epsilon_z]_{C^{0, \theta}(\C)} +([\mu]_{C^{0, \theta}(\C)})^{1+1/\theta}\|(\mu-\mu_\epsilon)f^\epsilon_z\|_\infty \right)\\
   & \leq C \|\mu-\mu_\epsilon\|_{C^{0,\theta}(\bar \Omega')} \|f^\epsilon\|_{C^{1,\theta}(\bar \Omega')} \leq C\|\mu-\mu_\epsilon\|_{C^{0,\theta}(\bar \Omega')} \to 0.\end{align*}
Since moreover $f^\epsilon \to 0 $ locally uniformly (cf. Lemma 5.3.5 in \cite{AIM2009}) it follows by interpolation that $f^\epsilon \to 0$ in $C^{1,\theta}_{loc}(\C)$. Appealing to the relation between $\psi$ and $\sigma$ in Definition \ref{de-conformal}, we also have $\sigma_\epsilon\to\sigma$ in $C^{0, \theta}_{loc}(\C)$ as $\epsilon\to0$, proving the claim.

\end{appendix}

\section*{Acknowledgement}
Wentao Cao's research was supported by Grant Agreement No. 12471224 of the National Natural Science Foundation of China.  He also thanks the hospitality of
Max Planck Institute for Mathematics in the Sciences (MPI MIS) since part of the work was completed when he was visiting MPI MIS.

\bibliographystyle{plain}

\end{document}